\documentclass[11pt, a4paper]{amsart}
\usepackage{graphicx}

\newtheorem{theorem}{Theorem}[section]
\newtheorem{proposition}[theorem]{Proposition}

\newtheorem{lemma}[theorem]{Lemma}
\newtheorem{corollary}[theorem]{Corollary}

\begin{document}
\baselineskip=16pt

\thispagestyle{empty}

\begin{center}\sf
{\Large Stability analysis for linear heat conduction with}\vskip.15cm
{\Large memory kernels described by Gamma functions}\vskip.2cm
{\small\tt GammaMem.tex}, \today\vskip.2cm
Corrado MASCIA\footnote{Dipartimento di
  Matematica ``G. Castelnuovo'', Sapienza -- Universit\`a di Roma, P.le Aldo
  Moro, 2 - 00185 Roma (ITALY), \texttt{\tiny mascia@mat.uniroma1.it}
  {\sc and} {Istituto per le Applicazioni del Calcolo, Consiglio Nazionale delle Ricerche
  (associated in the framework of the program ``Intracellular Signalling'')}}
\end{center}
\vskip.5cm


\begin{quote}\small
{\bf Abstract.}
This paper analyzes heat equation with memory in the case of kernels that are
linear combinations of Gamma distributions.
In this case, it is possible to rewrite the non-local equation as a local system 
of partial differential equations of hyperbolic type.
Stability is studied in details by analyzing the corresponding dispersion
relation, providing sufficient stability condition for the general case and 
sharp instability thresholds in the case of linear combination of the first three 
Gamma functions.
\end{quote}

\noindent
{\small
}
\vskip.25cm

\section{Introduction}
The present paper is based on three keywords: {\it diffusion}, {\it memory} and {\it stability}.
The starting ingredient is a scalar quantity $u$ satisfying
\begin{equation}\label{cons}
	\partial_t u(x,t) + \partial_x v(x,t)=0.
\end{equation}
where $x,t\in{\mathbb{R}}$ and $v$ is the {\it flux}.
Equation \eqref{cons} has to be completed with an additional relation for the unknown $v$.
The classical {\it Fourier law} (for heat conduction) prescribes that the flux is proportional to 
the space derivative of the unknown $u$,
\begin{equation}\label{fourierlaw}
	v(x,t)=-\alpha\,\partial_x u(x,t).
\end{equation}
The couple \eqref{cons}--\eqref{fourierlaw} gives raise to the well-known {\it heat equation},
which, being parabolic, has a number of appealing features (well-posedness, smoothing effects...)
At the same time, it also exhibits some deficiencies, the most striking being the infinite speed
of propagation of disturbances.
To mend such a flaw (and for other motivations which are not discussed here) alternative
laws for the flux $v$ has been proposed. 

The most famous, suggested by Cattaneo in \cite{Catt48}, assumes that the equilibrium between
flux and gradient, istantaneous in \eqref{fourierlaw}, is in fact described by a dynamical relation
of relaxation type, namely
\begin{equation}\label{catlaw}
	\tau\partial_t v(x,t)+v(x,t)+\alpha\,\partial_x u(x,t)=0
\end{equation}
where $\tau>0$ is an additional parameter describing the relaxation timescale.
The couple \eqref{cons}--\eqref{catlaw} gives raise to a {\it damped wave equation} for the unkwnon $u$.
Hyperbolicity of the model permits to trespass many of the flaws of the heat equation.

The Cattaneo's law \eqref{catlaw} can be rephrased into a more general framework 
in which the flux is not given by the {\it istantaneous} value of the gradient of $u$, but
it is rather an average of its history.
Precisely, equation \eqref{catlaw} is equivalent to the assumption that the flux $v$
satisfies the convolution relation 
\begin{equation}\label{memo}
	v(x,t)=-\int_{-\infty}^{t} g(t-s)\partial_x u(x,s)\,ds,
\end{equation}
with a kernel $g$ of exponential type, i.e. $g(t)=\alpha\,e^{-t/\tau}/\tau$.

As stated in \cite{JosePrez89}, {\it it would be a miracle if for some real conductor the relaxation kernel
could be rigorously represented by an exponential kernel with a single time of relaxation, as is required
by Cattaneo's model}.
Infact, the exponential choice can be considered as a stratagem which permits to trade the non-local
relation \eqref{memo} with the local relation \eqref{catlaw}, as a consequence of the property
\begin{equation}\label{cauchyexp}
	\tau \frac{dg}{dt}+g=0,\qquad \tau g(0)=\alpha.
\end{equation}
The conservation equation \eqref{cons} coupled with the {\it memory law} \eqref{memo}
gives raise to the non-local equation
\begin{equation}\label{memlaw}
	\partial_t u(x,t) -\int_{-\infty}^{t} g(t-s)\partial_{xx} u(x,s)\,ds = 0,
\end{equation}
which is meaningful for a class of {\it memory kernel} $g$ containing the pure exponentials
as a special case. 
The Fourier law \eqref{fourierlaw} can be recovered assuming that the kernel $g$ 
is a Dirac delta distribution concentrated at $t=0$.

Considering relation \eqref{memlaw} for modelling heat conduction has been proposed 
many years ago in \cite{GurtPipk68}, 
and it has received an increasing interest in recent years\footnote{Out of curiosity,
analyzing the database MathSciNet, article \cite{GurtPipk68}, published in 1968,
was cited for the first time in 1982, for the second time in 1998, and it has now more
than two hundred of citations, with an average of more than 15 citations for year in the period 2000-2014.}
(among others, see \cite{ChepMainPata06,ContMarcPata13,GiorNasoPata01} and references therein).
Usually, it is assumed that the function $g$ is a function with values in $[0,\infty)$ and such that
\begin{equation}\label{normg}
	\alpha:=\int_0^\infty g(s)\,ds<\infty.
\end{equation}
The value of $\alpha$ describes the intensity of the overall memory effect.

A key question is: {\it under which assumption on the memory kernel $g$, the dynamics is stable,
viz. there is no exponential growth of perturbations of constant states?}
In the two basic examples, Fourier \eqref{fourierlaw} and Cattaneo \eqref{catlaw} laws, 
stability holds as a consequence of a significant localization of the kernel $g$ at $t=0$.
How is the situation for memory kernels exhibiting a wider spreading in the past history?
Paraphrasing \cite{Fich79}, does having a lasting memory cause the appearance of nontrivial patterns?

Interesting answers has been given in \cite{ChepMainPata06,ContMarcPata13},
based on a semigroup theory approach.
The spirit is to determine properties on the function $g$ which guarantees stability
of the associated semigroup.
Over-simplifying (and suggesting to the interested reader to look at the originals),
the first article provides a necessary condition for stability, namely
$g'(t+s)\geq C\,e^{-\delta t}\,g'(s)$ for $t\geq 0$, $s>0$ for some $C\geq 1$, $\delta>0$,
which turns to be equivalent to $g+C\,g'\leq 0$ in $(0,\infty)$ for some $C>0$.
The second shows its sufficiency in the class of convex kernels $g$.

Here, the philosophy is different:
rather than considering general memory kernels, we analyze in details a specific class,
extending the pure exponentials, and still permitting a translation of the non-local term \eqref{memo}
in a set of local PDEs, so that the complete model turns to be equivalent to a linear hyperbolic system.
For such class of systems, we analyze the stability problem by means of a frequency analysis.

The form of the memory kernels is suggested by the fact that variations of the property \eqref{cauchyexp}
are satisfied by a more general class of functions: the {\it Gamma distributions}.
These form a two-parameter family defined as follows: given a positive real number $\tau$
and an integer $j\geq 1$, the Gamma distribution with {\it shape} $j$ and {\it scale} $\tau$ is
\begin{equation}\label{gammaktau}
	g(s;j,\tau):=\frac{1}{(j-1)!\tau^j}\,s^{j-1}\,e^{-s/\tau}.
\end{equation}
A sketch of the graph of the Gamma functions for the first values of $j$ (and $\tau=1$)
is depicted in Figure \ref{fig:gammagraph}.
Shape $j$ equal to $1$ corresponds to the exponential choice.
\begin{figure}
\includegraphics[width=9cm]{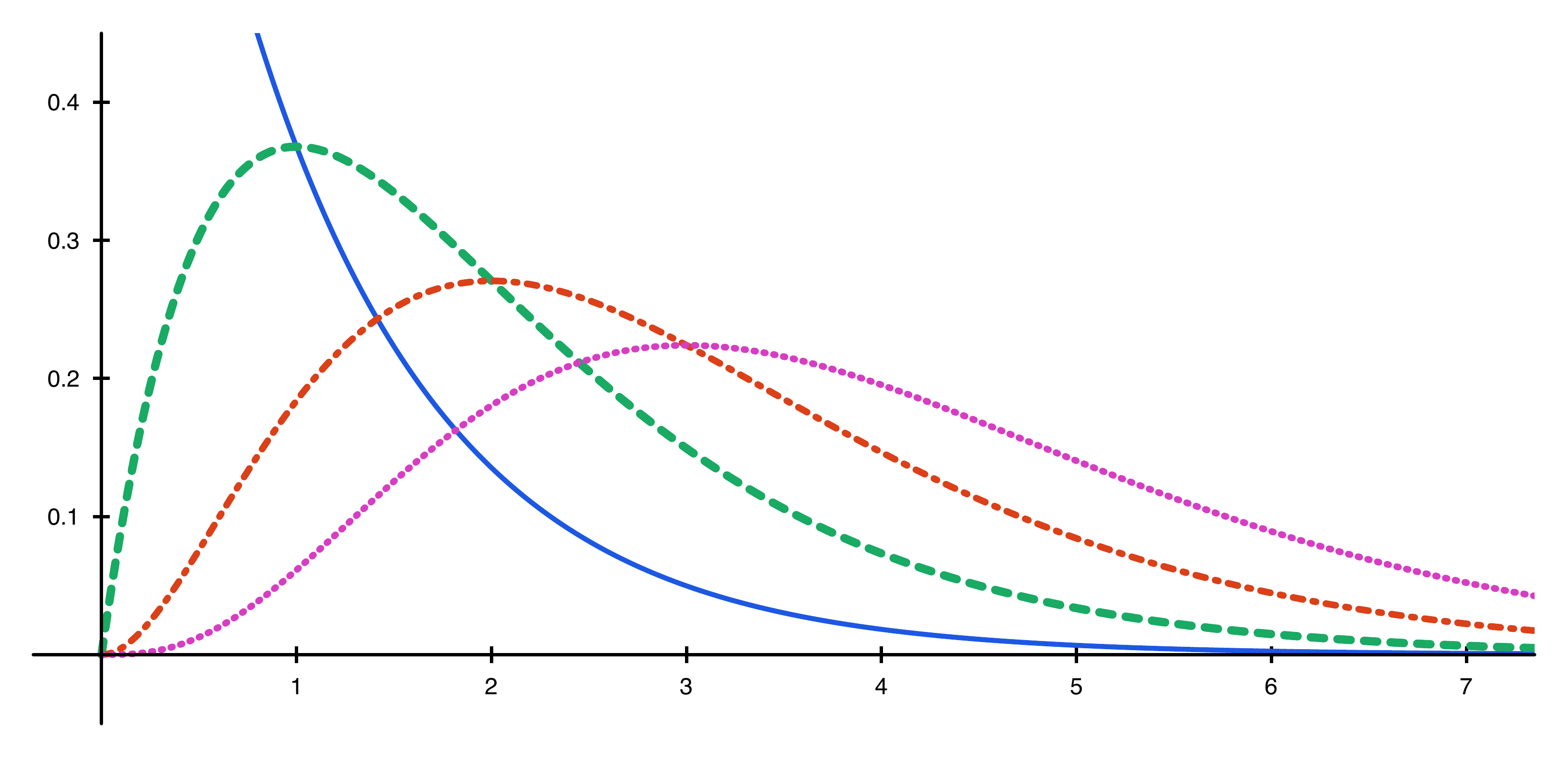}
\caption{\footnotesize Graph of the Gamma distributions for $\tau=1$ and $j=1$ (continuous, blue),
$j=2$ (dash, green), $j=3$ (dot-dash, red), $j=4$ (dot, magenta).}
\label{fig:gammagraph}
\end{figure}
Such functions are also defined for $j$ real,
but this case is not considered here for a reason that will be clear in a few lines.
In what follows, the functions $g(\cdot;j,\tau)$ will often be denoted by $g_j$, omitting the
dependence on $\tau$.

The choice of memory kernel given by Gamma distributions is common in the analysis of delay
differential equations, but, apparently more rare for PDEs.
Indeed, apart for the fact that different values $j$ correspond to different shape of the kernel
and, in particular, different dependencies on the past history in the diffusive dynamics, there is
no specific physical motivation for such a choice.

The main advantage is merely technical and it is encoded in what H. Smith calls the {\it linear chain trick}
(see \cite{Smit11}).
It consists in the fact that for such kernels it is possible to get rid of the memory term 
and to transform the non-local equation into a finite set of partial differential equations,
exactly as in the case of the exponential memory term relation \eqref{memlaw} is equivalent to
the hyperbolic system \eqref{cons}--\eqref{catlaw}.

\subsection*{Presentation of the main results.}
To enter into the details, given a natural number $k$, let $\mathbf{g}$ be the vector with components
$(g_1, g_2,\dots,g_{k+1})$ where $g_j$ are defined in \eqref{gammaktau}, by $\mathbb{I}_{k}$ the $k\times k$
identity matrix and by $\mathbb{N}_{k+1}$ the $(k+1)\times (k+1)$ matrix with $1$ on the subdiagonal
\begin{equation*}
	\mathbb{N}_{k+1}=\begin{pmatrix}	0_{1\times(k+1)} 	& 0	\\
							\mathbb{I}_{k} 		& 0_{(k+1)\times 1}
				\end{pmatrix}.
\end{equation*}
Then, there holds 
\begin{equation*}
	\tau\frac{d\mathbf{g}}{dt}+\bigl(\mathbb{I}_{k+1}-\mathbb{N}_{k+1}\bigr) \mathbf{g}=0,\qquad
	\tau\mathbf{g}(0)=\mathbf{e}_1,
\end{equation*}
where $\mathbf{e}_1=(1,0,\dots,0)$, to be considered as an extension of \eqref{cauchyexp}.
As a consequence, for any kernel $g$ given by linear combinations of Gamma distributions, it is possible
to convert the convolution relation \eqref{memo} into a set of differential equations for some auxiliary
functions, whose linear combination, furnishes the flux $v$ given in \eqref{cons}.
Precisely, given $\tau>0$, $k\geq 1$ integer and $\boldsymbol{\theta}=(\theta_1,\dots,\theta_{k+1})\in{\mathbb{R}}^{k+1}$,
if the memory kernel $g$ is given by
\begin{equation}\label{kernelform}
	g(t)=\boldsymbol{\theta}\cdot \mathbf{g}(t)=\sum_{j=1}^{k+1} \theta_j g(t;j,\tau),
\end{equation}
then the non-local equation \eqref{memlaw} is equivalent to a system for $(u,\mathbf{v})\in{\mathbb{R}}^{k+2}$ 
\begin{equation}\label{main}
	\partial_t u + \partial_x \bigl(\boldsymbol{\theta}\cdot \mathbf{v}\bigr)=0,\qquad
	\tau\partial_t \mathbf{v}+\mathbf{e}_1\partial_x u
			+\bigl(\mathbb{I}_{k+1}-\mathbb{N}_{k+1}\bigr)\mathbf{v}=0.
\end{equation}
System \eqref{main}, which is a generalization of the Cattaneo diffusion equation
(recovered for $k=0$), turns to be (non strictly) hyperbolic whenever $\theta_1>0$
(see Section \ref{sec:chaintrick}).
The choice \eqref{kernelform} for the kernel $g$ allows the presence of a single time-scale $\tau$
which is from now on normalized to $1$.
The compresence of multiple time scales is meaningful, but it is not considered here.

The goal of this paper is to analyze in details the stability/instability properties of \eqref{main} 
for different ranges of the parameter $\boldsymbol{\theta}$, by means of a detailed description
of the {\it dispersion relation}
\begin{equation}\label{disprel0}
	p(\lambda,\mu):=\det\begin{pmatrix}
		\lambda u 			&	\mu \boldsymbol{\theta} \\
		\mathbf{e}_1\mu	&	(\lambda+1)\mathbb{I}_{k+1}-\mathbb{N}_{k+1}
		\end{pmatrix}=0
\end{equation}
Any couple $(\lambda, \mu)$ such that \eqref{disprel} holds corresponds to a solution of \eqref{main}
with the form $(u,\mathbf{v})=(U,\mathbf{V})e^{\lambda t+\mu x}$.
Then, the stability character of the model \eqref{main}  is encoded in the sign of the real part of the
elements of the sets $\Lambda(\xi)$ for $\xi\in{\mathbb{R}}$, where 
\begin{equation*}
	\Lambda(\xi):=\{\lambda\in{\mathbb{C}}\,:\,p(\lambda, i\xi)=0\}.
\end{equation*} 
Such Fourier analysis can be translated into the description of the $L^2-$decay in time of solutions
with data in $L^2\cap L^1$ following a standard procedure, based on the use of inverse Fourier transform.

General results concerning the behavior of the dispersion relation for small and large
frequencies are encoded in the following statement.

\begin{theorem}\label{thm:smalllarge}
Let the memory kernel $g$ in \eqref{memo} be of the form $\boldsymbol{\theta}\cdot \mathbf{g}$
for some $\boldsymbol{\theta}\in{\mathbb{R}}^{k+1}$ with $\theta_1>0$ and $\theta_j\geq 0$ for any $j$.\par
{\bf I.} {\sl [small frequencies]} There exist $r,\delta>0$ such that for $|\xi|<r$ the set
$\Lambda(\xi)\cap\{\lambda\,:\, {\textrm{Re}}\lambda\geq -\delta\}$ is composed by a single
branch $\lambda=\lambda_0(\xi)$ and there holds
\begin{equation*}
	\lambda_0(\xi)=-\left(\boldsymbol{\theta}\cdot \mathbf{1}\right)\xi^2+o(\xi^2)
		\qquad\textrm{as}\quad \xi\to 0,
\end{equation*}
where $\mathbf{1}=(1,\dots,1)\in{\mathbb{R}}^{k+1}$.\par
{\bf II.} {\sl [large frequencies]} There exist $R,\delta>0$ such that for $|\xi|>R$ there holds
$\Lambda(\xi)\subset\{\lambda\,:\,{\textrm{Re}}\lambda\leq -\delta\}$
if and only if
\begin{equation}\label{eta2small}
	\theta_2<\theta_1
\end{equation}
and all of the roots of the polynomial
\begin{equation}\label{HFpoly}
	Q(x)=\sum_{\ell=0}^{k} q_\ell x^\ell\qquad\textrm{where}\quad
	q_\ell=\sum_{j=1}^{k+1-\ell}\binom{k+1-j}{\ell}\theta_{j}.
\end{equation}
have strictly negative real parts.
\end{theorem}

Part {\bf II.} of Theorem \ref{thm:smalllarge} encodes high-frequency stability in polynomial $Q$.
As an example, in the case $k=1$, there holds $Q(x)=\theta_{1}+\theta_{2}+\theta_1 x$
with the root $-1-\theta_2/\theta_1<0$.
If $k=2$, the polynomial $Q$, given by
\begin{equation*}
	Q(x)=\theta_{1}+\theta_{2}+\theta_{3}+(2\theta_1+\theta_2) x+\theta_1 x^2,
\end{equation*}
has roots with real part $-1-\theta_2/2\theta_1<0$
In both cases, large frequencies stability holds if and only if condition \eqref{eta2small} is satisfied.

If both small and large frequencies stability holds, in order to estabilish the complete dynamical
property of the system, it is necessary to analyze the behavior of the dispersion relation for
intermediate values of the frequency.
Specifically, it is relevant to analyze the existence/non existence of couples $(\lambda,\mu)$
with the form $(i\zeta,i\xi)$ such that \eqref{disprel0} is satisfied.
In case of non-existence of such couples, collecting with the properties that holds in case
of small and large frequencies stability, one can state that there exists $c_0>0$ such that
\begin{equation}\label{stability}
	{\textrm{Re}}\lambda \leq -\frac{c_0|\xi|^2}{1+|\xi|^2}
	\qquad\forall\xi\in{\mathbb{R}},\,\lambda\in\Lambda(\xi).
\end{equation}
The validity of such estimate guarantees that the
{\it Shizuta--Kawashima condition} is satisfied and thus the model 
exhibits dissipation (see \cite{ShizKawa85} and descendants).

This paper furnishes two additional results in this direction.
The first one gives sufficient condition for stability for general values of $k$.

\begin{theorem}\label{thm:anyk}
Let the memory kernel $g$ in \eqref{memo} be of the form $\boldsymbol{\theta}\cdot \mathbf{g}$
for some $\boldsymbol{\theta}\in{\mathbb{R}}^{k+1}$ with $\theta_1>0$ and $\theta_j\geq 0$ for any $j$.
If the coefficients $\theta_j$ are such that
\begin{equation}\label{condanyk}
	\sum_{j=2}^{k+1}\theta_j<\theta_1,
\end{equation}
then there exists $c_0>0$ such that \eqref{stability} holds.
\end{theorem}

Some of the kernels $g$ for which the hypotheses of Theorem \ref{thm:anyk} hold
change convexity for $t\in (0,\infty)$.
Indeed, for $k=1$ there hold
\begin{equation}
	g(t)=(\theta_1+\theta_2 t)e^{-t}
	\qquad\textrm{and}\qquad
	g''(t)=(\theta_1-2\theta_2+\theta_2 t)e^{-t},
\end{equation}
so that, when $\theta_1<2\theta_2$,
$g''$ is negative in $(0,2-\theta_1/\theta_2)$ and positive otherwise.
In particular, for $\theta_2<\theta_1<2\theta_2$, the kernel $g$ is non-convex and nevertheless
determine a stable dynamics.
Similar properties hold for general $k$ and $\theta_j$ small for $j>2$.

Let us observe that memory kernel derived by homegenization as in \cite{ChowChri10},
while exhibiting decreasing monotonicity, appear to be non-convex.

The second result is specific to the case $k=2$ and furnishes a complete characterization
for the values of $\boldsymbol{\theta}$ guaranteeing stability.

\begin{theorem}\label{thm:k=2}
Let $k=2$, so that
\begin{equation*}
	g(t)=\left(\theta_1+\theta_2 t+\frac12\theta_3 t^2\right)e^{-t}.
\end{equation*}
Then, assuming $\theta_1>0$ and $\theta_2, \theta_3\geq 0$,
estimate \eqref{stability} holds if and only if
\begin{equation}\label{condk=2}
	\theta_2<\theta_1
	\qquad\textrm{and}\qquad
	\left\{\begin{aligned}
	&\textrm{either}\quad 3\theta_3<2\theta_1,\\
	&\quad\textrm{or}\quad
	(2\theta_2+\theta_3)^2+8(\theta_3-\theta_1)^2<8\theta_1^2.
	\end{aligned}\right.
\end{equation}
\end{theorem}

The conditions on the parameters $\theta_1,\theta_2,\theta_3$ define three regions
corresponding to three different behaviors.
Setting $\eta_2=\theta_2/\theta_1$ and $\eta_3=\theta_3/\theta_1$,
the region of high frequency stability is $\eta_2<1$.
In this strip, the set where condition \eqref{stability} holds is
\begin{equation*}
	\{(\eta_2,\eta_3)\,:\, \eta_3<2/3\}\cup
	\{(\eta_2,\eta_3)\,:\, (2\eta_2+\eta_3)^2+8(\eta_3-1)^2<8\},
\end{equation*}
where the second set describes the interior of an ellipsis
(see Figure \ref{fig:stabregions1}).
\begin{figure}
\includegraphics[width=10cm]{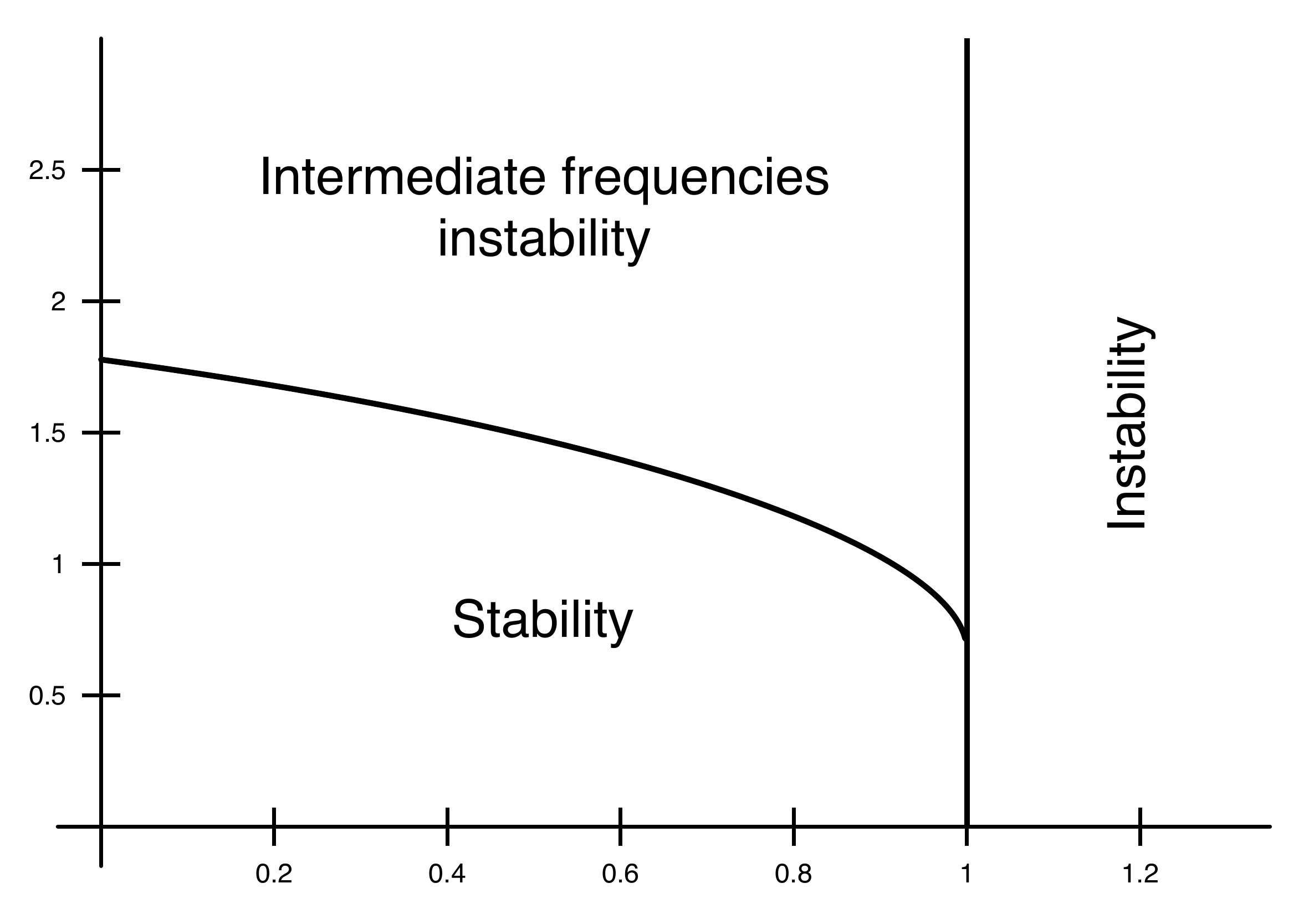}
\caption{\footnotesize Case $k=2$.
The stability character of the model In the plane $(\eta_2,\eta_3)$ (where $\eta_j=\theta_j/\theta_1$).}
\label{fig:stabregions1}
\end{figure}

Theorem \ref{thm:k=2} indicates the existence of a parameter region where the instability
is of ``Turing type'', in the sense that large and small frequencies are damped by the equation
while a window of intermediate frequencies exhibit exponential growth in time. 
In this respect, this paper intends to contribute to the strand of the analysis of instabilities
driven by the presence of memory (see \cite{ODRK86, DuffFreiGrin02} and descendants).

The organization of the paper is the following.
Section \ref{sec:chaintrick} starts by discussing the linear chain trick for Gamma distributions.
Then, it concerns with the hyperbolicity of the system \eqref{main} in term of the coefficients $\theta_j$.
In Section \ref{sec:smalllarge}, we analyze the dispersion relation of the system dividing the
small and large frequencies regimes, giving the complete proofs of 
Theorems \ref{thm:smalllarge} and 
Finally, in Section \ref{sec:inter}, we analyze intermediate frequencies
giving the proof of Theorem \ref{thm:anyk} and then concentrating on
the case $k=2$, providing the proof of Theorem \ref{thm:k=2}.

\section{Linear chain trick and hyperbolicity}\label{sec:chaintrick}

For the reader's convenience, following closely \cite{Smit11}, this Section starts with the
presentation of the {\it linear chain trick} and its use to reduce the diffusion equation \eqref{cons}
with the flux $v$ defined in \eqref{memo} to a system of partial differential equations,
in the case of memory kernels given by a linear combination of
Gamma functions \eqref{gammaktau}.

Fixed the scale $\tau$ (normalized to 1, without loss of generality) and given $T>0$,
let us consider the linear functionals
\begin{equation}\label{trasfk}
	\phi\;\mapsto\;\psi_{j}=\mathcal{T}_{j}\phi
	\qquad\textrm{where}\quad
	\psi_{j}(t):=\int_{-\infty}^{t} g_j(t-s)\phi(s)\,ds,
\end{equation}
defined for bounded functions $\phi\in C((-\infty,T];{\mathbb{R}})$.

\begin{proposition}\label{prop:lintrick}
Let $T\in{\mathbb{R}}$.
Given $\psi_0\in C((-\infty,T];{\mathbb{R}})$, the system of ordinary differential equations
for the unknowns $\psi_j=\psi_j(t)$
\begin{equation*}
	\frac{d\psi_j}{dt}=\psi_{j-1}-\psi_{j}
	\qquad j=1,2,\dots
\end{equation*}
has a unique bounded solution $\psi_1,\psi_2,\dots$.
Such solution is given by $\psi_j=\mathcal{T}_{j}\psi_0$ for any $j=1,2,\dots$,
with $\mathcal{T}_j$ defined in \eqref{trasfk}.
\end{proposition}

\begin{proof}
For $j\geq 2$, integrating by parts, we infer
\begin{equation*}
		I_j:=\int_{0}^{\infty} g_j (t)\,dt=\frac{1}{(j-1)!} \int_{0}^{\infty} s^{j-1}\,e^{-s}\,ds
			=\frac{1}{(j-2)!} \int_{0}^{\infty} s^{j-2}\,e^{-s}\,ds=I_{j-1}.
\end{equation*}
Thus, since $I_1=1$, there holds $I_j=1$ for any $j$.
Therefore we infer
\begin{equation*}
	\left|\mathcal{T}_{j}\phi\right|\leq \int_{-\infty}^{t} g_j(t-s)|\phi(s)|\,ds
		\leq \int_{0}^{\infty} g_j(s)\,ds\,|\phi|_{{}_{L^\infty}} =|\phi|_{{}_{L^\infty}}
\end{equation*}
which gives
\begin{equation*}
	\left|\mathcal{T}_{j}\phi\right|_{{}_{L^\infty}}\leq |\phi|_{{}_{L^\infty}}.
\end{equation*}
Moreover, each $g_j$ solves an appropriate Cauchy problem, viz.
\begin{equation*}
	\left\{\begin{aligned} \frac{dg_1}{dt}&=-g_1,\\ g_1(0)&=1 \end{aligned}\right.
	\qquad\textrm{and}\qquad
	\left\{\begin{aligned} \frac{dg_j}{dt}&=g_{j-1}-g_j,\\ g_j(0)&=0, \end{aligned}\right.
	\quad j=2,3,\dots.
\end{equation*}
Then, differentiating the definition $\psi_1:=\mathcal{T}_{1}\psi_0$ we obtain
\begin{equation*}
	\begin{aligned}
	\frac{d\psi_1}{dt}
		&=g_1(0)\psi_0(t)+\int_{-\infty}^{t} \partial_t g_1(t-s)\psi_0(s)\,ds\\
		&=\psi_0(t)-\int_{-\infty}^{t} g_1(t-s)\psi_0(s)\,ds=\psi_0(t)-\psi_1(t).
	\end{aligned}
\end{equation*}
Similarly, differentiating  $\psi_j:=\mathcal{T}_j\psi_0$ for $j\geq 2$, we infer
\begin{equation*}
	\begin{aligned}
	\frac{d\psi_j}{dt}&=g_j(0)\psi_0(t)+\int_{-\infty}^{t} \partial_t g_j(t-s)\psi_0(s)\,ds\\
		&=\int_{-\infty}^{t} \bigl\{g_{j-1}(t-s)-g_j(t-s)\bigr\}\psi_0(s)\,ds=\psi_{j-1}-\psi_{j}.
	\end{aligned}
\end{equation*}
To conclude the proof, it has to be shown that the problem has at most one solution,
viz. the one just exhibited.
Assuming that $\psi_j^1$ and $\psi_j^2$ both solve the prescribed linear system,
the difference $\Psi_j:=\psi_j^1-\psi_j^2$ is a bounded solution to
\begin{equation*}
	\frac{d\Psi_1}{dt}=-\Psi_{1},\quad
	\frac{d\Psi_j}{dt}=\Psi_{j-1}-\Psi_{j}\quad j=2,\dots.
\end{equation*}
Then, $\Psi_1$ is clearly zero and, by induction, also $\Psi_j$ is zero for any $j\geq 2$.
\end{proof}

Thanks to Proposition \ref{prop:lintrick}, the non-local problem
\begin{equation*}
	\partial_t u + \partial_x v=0\qquad\textrm{with}\quad
	v(x,t):=-\int_{-\infty}^{t} g(t-s)\partial_x u(x,s)\,ds
\end{equation*}
can be reduced to a system of partial differential equations when the memory kernel
$g$ has the form \eqref{kernelform} for some natural $k$ and 
$\boldsymbol{\theta}\in{\mathbb{R}}^{k+1}$.
Indeed, there holds
\begin{equation*}
	v=-\sum_{j=1}^{k+1} \theta_j \mathcal{T}_j(\partial_x u)=-\sum_{j=1}^{k+1} \theta_j \psi_j,
\end{equation*}
so that the problem is equivalent to the system
\begin{equation}\label{gensys}
	\left\{\begin{aligned}
		&\partial_t u + \theta_{1} \partial_x\psi_{1} + \dots +\theta_{k+1} \partial_x\psi_{k+1} = 0,\\
		&\partial_t \psi_1 + \partial_x u + \psi_1 = 0 ,\\
		&\partial_t \psi_j 	- \psi_{j-1} + \psi_{j} = 0 \qquad\quad j=2,\dots,n+1
	\end{aligned}\right.
\end{equation}
System \eqref{gensys} can be rewritten in vectorial form as
\begin{equation}\label{vector}
	\mathbb{A}_0 \partial_t W + \mathbb{A}_1 \partial_x W + \mathbb{B} W = 0,
\end{equation}
where $W=(u,\psi_1,\dots,\psi_{k+1})$ and 
\begin{equation}\label{A0A1}
	\mathbb{A}_0=\begin{pmatrix}
		1			& 0 			& 0_{k\times 1} \\
		0 			& 1			& 0_{k\times 1} \\
		0_{1\times k} 	& 0_{1\times k} 	& \mathbb{I}_{k}
	\end{pmatrix},
		\qquad
	\mathbb{A}_1=\begin{pmatrix}
		0			& \theta_1		& \theta_\ast \\
		1			& 0	 		& 0_{k\times 1} \\
		0_{1\times k} 	& 0_{1\times k} & 0_{k\times k}
	\end{pmatrix}
\end{equation}
where $\theta_\ast=(\theta_2,\dots,\theta_{k+1})$ and
\begin{equation}\label{B}
	\mathbb{B}=\begin{pmatrix}
		0 & 0 & 0 & \dots & 0 \\
		0 & 1 & 0 & \dots & 0 \\
		0 & -1 & 1  & \dots & 0 \\
		\vdots & \vdots & \vdots & \ddots & \vdots \\
		0 & 0 & 0 & \dots & 1 \\
	\end{pmatrix}
\end{equation}
The hyperbolicity character of the system is recognized analyzing the principal term
in \eqref{vector} and, specifically, the spectral properties of the matrix
$\mathbb{A}_0^{-1}\mathbb{A}_1$.

\begin{lemma}\label{lem:hyperb}
The matrix $\mathbb{A}_0^{-1}\mathbb{A}_1$ has real eigenvalues if and only if $\theta_1\geq 0$.
Moreover, the matrix $\mathbb{A}_0^{-1}\mathbb{A}_1$\par
{\bf i.} if $\theta_1=0$, has a single eigenvalue $0$ with algebraic multiplicity equal to $n+2$
and geometric multiplicity equal to $n$;\par
{\bf ii.} if $\theta_1>0$, has three eigenvalues: $\pm\sqrt{\theta_1}$ (each with multiplicity equal to 1)
and $0$ (with algebraic and geometric multiplicity equal to $k$).
\end{lemma}

\begin{proof}
The eigenvalue problem is $(\mathbb{A}_1-\lambda \mathbb{A}_0)W=0$.
Therefore, since
\begin{equation*}
	\det(\mathbb{A}_1-\lambda \mathbb{A}_0)
		=(-\lambda)^{k}\det\begin{pmatrix} -\lambda & \theta_1 \\ 1 & -\lambda \end{pmatrix}
		=(-\lambda)^{k}(\lambda^2-\theta_1),
\end{equation*}
the statement on the eigenvalues and the algebraic multiplicity holds.
Moreover, for $\lambda=0$, the eigenvalue problem reduces to the conditions
\begin{equation*}
	 u = 0,\qquad \sum_{k=1}^{k+1} \theta_k \psi_k = 0,
\end{equation*}
which define a vector space of dimension $k$ in ${\mathbb{R}}^{k+2}$.
\end{proof}

As a consequence of Lemma \ref{lem:hyperb}, it is reasonable to concentrate the attention
on the case $\theta_1>0$, since only in such a case the principal part of the system is
diagonalizable and thus the Cauchy problem is well-posed.

A diagonalizing matrix can be determined by computing the $k+2$ independent eigenvectors.
Since the eigenvalue problem reads as
\begin{equation*}
		-\lambda u + \sum_{j=1}^{k+1} \theta_j  \psi_j = 0,\qquad
		-\lambda \psi_1 + u = 0 ,\qquad
		\lambda \psi_j 	= 0 \quad (j=2,\dots,k+1),
\end{equation*}
a choice of eigenvectors is
\begin{equation*}
	\begin{aligned}
	\pm\sqrt{\theta_1} :\quad & u=1,\quad \psi_1=\pm 1/\sqrt{\theta_1}, &\quad \psi_j&=0,\\
	0 :\quad & u=0,\quad	\psi_1=-\theta_i,	&\quad \psi_j&=\theta_1\delta_{ij}\quad (i=2,\dots,n+1)
	\end{aligned}
\end{equation*}
where $\delta_{ij}$ is the Kronecker symbol.
Such eigenvectors, considered as column vectors, define a change of coordinates $\mathbb{C}$
such that, setting $Z=\mathbb{C}^{-1} W$, the principal part of the system \eqref{vector} is diagonalized
\begin{equation*}
	\partial_t Z + \mathbb{D}\,\partial_x Z + \mathbb{E}\,Z = 0,
\end{equation*}
where $\mathbb{D}:={\textrm{diag}}(-\sqrt{\theta_1},\sqrt{\theta_1},0,\dots,0)$ and
$\mathbb{E}:=\mathbb{C}^{-1}\mathbb{A}_0^{-1}\mathbb{B}\,\mathbb{C}$.

\section{Small and large frequencies analysis}\label{sec:smalllarge}

Applying Fourier--Laplace transform, system \eqref{vector} is converted into the system
\begin{equation*}
	\bigl(\lambda\mathbb{A}_0  + \mu \mathbb{A}_1 + \mathbb{B} \bigr) W = 0,
\end{equation*}
which possesses non-trivial solution if and only if the {\it dispersion relation}
\begin{equation}\label{dispersion}
	p(\lambda,\mu):=\det\bigl(\lambda\mathbb{A}_0  + \mu \mathbb{A}_1 + \mathbb{B} \bigr)=0
\end{equation}
is satisfied.
Such relation can be read in different ways.
Here, following the standard approach dictated by the Fourier transform, the variable
$\mu$ will be considered as a purely imaginary number, i.e. $\mu=i\xi$ with $\xi\in{\mathbb{R}}$,
and the attention is directed toward the corresponding values of $\lambda\in{\mathbb{C}}$ such that
\eqref{dispersion} is satisfied.

Thanks to the very special structure of the matrices defining the polynomial $p$,
it is possible to determine an explicit formula for it.

\begin{lemma}\label{lem:disprel}
Let $\mathbb{A}_0, \mathbb{A}_1$ and $\mathbb{B}$ as in \eqref{A0A1}--\eqref{B}.
Then, there holds
\begin{equation}\label{disprel}
	p(\lambda,\mu)=\lambda(\lambda+1)^{k+1}
		-\mu^2\sum_{j=0}^{k} \theta_{k+1-j}(\lambda+1)^{j}.
\end{equation}
\end{lemma}

\begin{proof}
The formula can be proved by induction on $k$.
Indeed, for $k=0$ the relation 
\begin{equation*}
	p(\lambda,\mu)=\lambda(\lambda+1)-\theta_1\mu^2
\end{equation*}
holds, as can be seen by direct computation.

Then, making explicit the dependence on $n$ in the dispersion relation
denoting the polynomial by $p_k$, there holds
\begin{equation*}
	\begin{aligned}
	p_{k}(\lambda,\mu)&=(\lambda+1)p_{k-1}(\lambda,\mu)
		+(-1)^{k+1}\theta_{k+1}\mu
		\det\begin{pmatrix}
			\mu 		& \lambda+1	& \dots 	& 0	\\
			0 		& -1  			& \dots 	& 0	\\
			\vdots	& \vdots			& \ddots 	& \lambda +1 \\
			0		& 0 				& 0		& -1 \\
			\end{pmatrix}\\
		&=(\lambda+1)p_{k-1}(\lambda,\mu)-\theta_{k+1}\mu^2.
	\end{aligned}
\end{equation*}
Then, assuming \eqref{disprel} for $k-1$, we infer
\begin{equation*}
	\begin{aligned}
	p_{k}(\lambda,\mu) &=(\lambda+1)p_{k-1}(\lambda,\mu)-\theta_{k+1}\mu^2 \\
		& =\lambda(\lambda+1)^{k+1}
		-\mu^2\sum_{j=1}^{k} \theta_j(\lambda+1)^{k+1-j}
		-\theta_{k+1}\mu^2 \\
		& =\lambda(\lambda+1)^{k+1}
		-\mu^2\sum_{j=1}^{k+1} \theta_j(\lambda+1)^{k+1-j}.
	\end{aligned}	
\end{equation*}
The proof is complete.
\end{proof}

The analysis of the relation $p(\lambda,\mu)=0$ gives a complete information on the 
stability properties in the model. 
First of all, the behavior for small values of $\mu$, corresponding to the large time-behavior,
in case of stability, is very simple.
Indeed, setting $\mu=0$, there holds
\begin{equation*}
	p(\lambda,0)=\lambda(\lambda+1)^{k+1}.
\end{equation*}
Thus, there are $k+1$ modes that are dissipated with exponential rate (corresponding 
to the zero $-1$ with order $k+1$) and there is a single mode with a weaker decay.

In order to explore in details the slow-mode, relative to the value $\lambda=0$,
the expansion $\lambda=A\mu+B\mu^2+o(\mu^2)$
can be plugged into the dispersion relation, obtaining
\begin{equation*}
	\begin{aligned}
	0=p(\lambda,\mu)
	&=\bigl\{A\mu+B\mu^2+o(\mu^2)\bigr\}\bigl\{1+A\mu+o(\mu)\bigr\}^{k+1}
		-\mu^2\sum_{j=0}^{k} \theta_{k+1-j}+o(\mu^2)\\
	&=\bigl\{A\mu+B\mu^2+o(\mu^2)\bigr\}\bigl\{1+(k+1)A\mu+o(\mu)\bigr\}
		-\left(\boldsymbol{\theta}\cdot \mathbf{1}\right)\mu^2+o(\mu^2)\\
	\end{aligned}
\end{equation*}
that is
\begin{equation*}	
	0=p(\lambda,\mu)
	=A\mu+\Bigl\{(k+1)A^2+B-\left(\boldsymbol{\theta}\cdot \mathbf{1}\right)\Bigr\}\mu^2+o(\mu^2).
\end{equation*}
Thus, in a neighborhood of $(\lambda,\mu)\approx (0,0)$, the dispersion relation defines the curve
\begin{equation*}
	\lambda=\left(\boldsymbol{\theta}\cdot \mathbf{1}\right)\mu^2+o(\mu^2)
\end{equation*}
This completes the proof of Theorem \ref{thm:smalllarge}, part {\bf I}.

Next, we investigate the behavior of the couples $(\lambda,\mu)$ satisfying the dispersion
relation, in the regime $\mu\to\infty$.
Keeping track of the highest powers of $\lambda$ and $\mu$ gives
\begin{equation*}
	p(\lambda,\mu)=(\lambda^{2}-\theta_1\mu^2)\lambda^{k}+\dots 
\end{equation*}
Thus, as $\mu\to\infty$, setting $\kappa:=\sqrt{{\theta_1}}$,
the dispersion relation defines a set separated into three branches,
given by the asymptotics
\begin{equation*}
	\lambda=\pm \kappa\,\mu+o(\mu),\qquad \lambda=o(\mu),
\end{equation*}
which corresponds to the eigenvalues of the principal part of the system \eqref{vector}.

The stability properties of such branches are encoded in the subsequent term in the asymptotic
expansions for  $\lambda$ as a function of $\mu$.
Setting $\lambda=\pm \kappa\,\mu+C+o(1)$ with $C$ to be determined,
and plugging into the dispersion relation, one obtains
\begin{equation*}
	\begin{aligned}
	0&=\bigl\{\pm \kappa\,\mu+C+o(1)\bigr\}\bigl\{\pm\kappa\,\mu+C+1+o(1)\bigr\}^{k+1}
		-\mu^2\theta_{2}\bigl\{\pm\kappa\,\mu+C+1+o(1)\bigr\}^{k-1}\\
	&\hskip6.5cm -\mu^2\theta_{1}\bigl\{\pm\kappa\,\mu+C+1+o(1)\bigr\}^{k}+o(\mu^{k+1})
	\end{aligned}
\end{equation*}
which reduces to
\begin{equation*}
	\begin{aligned}
	0&=(\pm \kappa)^{k-1}\Bigl\{\Bigl((n+2)\kappa^2-k\theta_1\Bigr)C
		+(k+1)\kappa^2-\theta_2-k\theta_1\Bigr\}\mu^{k+1}+o(\mu^{k+1})\\
		&=(\pm \kappa)^{k-1}\bigl(2\theta_1\,C+\theta_1-\theta_2\bigr)\mu^{k+1}+o(\mu^{k+1})
	\end{aligned}
\end{equation*}
so that
\begin{equation*}
	C=-\frac{\theta_1-\theta_2}{2\theta_1}.
\end{equation*}
Similarly, choosing $\lambda=C+o(1)$
and inserting in the dispersion relation, one infers
\begin{equation*}
		-\mu^2\sum_{j=0}^{k} \theta_{k+1-j}(C+1)^{j}+o(\mu^2)=0
\end{equation*}
so that the coefficient $C$ satisfies the algebraic relation
\begin{equation*}
		\sum_{j=0}^{k} \theta_{k+1-j}(C+1)^{j}=0.
\end{equation*}
Such equality can be rewritten as
\begin{equation*}
	\sum_{j=0}^{k}\sum_{\ell=0}^{j}  \binom{j}{\ell}\theta_{k+1-j} C^\ell=0.
\end{equation*}
Exchanging the order of the sums, it follows that $C$ is root of the polynomial
\begin{equation*}
	Q(x)=\sum_{\ell=0}^{k} q_\ell x^\ell\qquad\textrm{where}\quad
	q_\ell:=\sum_{j=\ell}^{k}\binom{j}{\ell}\theta_{k+1-j}.
\end{equation*}
This completes the proof of Theorem \ref{thm:smalllarge}, part {\bf II}.

Rephrasing, necessary and sufficient conditions for stability in the high frequency regime
are $\theta_2<\theta_1$ and the request that the polynomial $Q$ is
a {\it (strict) Hurwitz polynomial}, i.e. a polynomial whose roots have strictly negative real parts.
Verification of the latter condition can be done by using the {\it Routh--Hurwitz stability condition}.
Let us consider some explicative examples, when $\theta_1+\dots+\theta_{k+1}=1$.
We list the explicit form of the polynomial $Q$ and the corresponding
Routh--Hurwitz array for the first integers and for specific choices of
the parameters $\theta_j$:\\
for $k=2$ and general $\theta_j$,
\begin{equation*}
	Q(x)=1+(1+\theta_1) x+\theta_1 x^2
	\qquad\qquad
	\begin{tabular}{@{}c|c|c@{}}  
		$\theta_1$	& 1	& 0	\\ \hline
		$1+\theta_1$	& 0	&	\\ \hline
		1			&	& 
	\end{tabular}
\end{equation*}
for $k=3$, with $\theta_2=\theta_3=0$,
\begin{equation*}
	Q(x)=1+3\theta_1 x+3\theta_1 x^2+\theta_1 x^3
	\qquad\qquad
	\begin{tabular}{@{}c|c|c|c@{}}  
		$\theta_1$		& $3\theta_1$ 	& 0	& 0	\\ \hline
		$3\theta_1$		& 1			& 0	&	\\ \hline 
		$3(\theta_1-1/9)$	& 0			&	&	\\ \hline
		1 				& 			& 	&
	\end{tabular}
\end{equation*}
for $k=4$, with $\theta_2=\theta_3=\theta_4=0$,
\begin{equation*}
	Q(x)=1+4\theta_1 x+6\theta_1 x^2+4\theta_1 x^3+\theta_1 x^4
	\qquad\qquad
	\begin{tabular}{@{}c|c|c|c|c@{}}  
		$\theta_1$		& $6\theta_1$ 	& 1	& 0	&	0	\\ \hline
		$4\theta_1$		& $4\theta_1$	& 0	& 0	\\ \hline 
		$5\theta_1$		& 1			& 0	&	\\ \hline
		$4(\theta_1-1/5)$ 	& 0			& 	&	\\ \hline
		1			 	& 			& 	&
	\end{tabular}
\end{equation*}
In particular, large frequency stability holds for $k=2$ (any case), $k=3$ and $\theta_1>1/9$, $k=4$
and $\theta_1>1/5$.

For higher values of $k$, the necessary and sufficient conditions appear difficult to be translated
in a simple condition that is directly readable from the values of the coefficients $\theta_k$.
Nevertheless, for $\theta_1$ sufficiently large with respect to the other coefficients $\theta_j$,
$j=2,\dots,k+1$, the stability condition is satisfied.

\begin{corollary}\label{coro:suffhighstab}
Let $k\geq 2$ and let $\theta_1,\dots,\theta_{k+1}$ be such that
\begin{equation}\label{bigtheta1}
	 \sum_{j=2}^{k+1} \theta_{j}\leq \theta_1.
\end{equation}
Then the system \eqref{gensys} is stable in the high-frequency regime.
\end{corollary}

\begin{proof}
The polynomial $Q$ can be rewritten as
\begin{equation*}
	Q(x)=\sum_{j=0}^{k} \theta_{k+1-j}(1+x)^{j}
\end{equation*}
In the case $\theta_j=0$ for any $j\neq 1$, the polynomial reduces to $(1+x)^k$
and, thus, it has a single root $x=-1$ with algebraic multiplicity equal to $k$.
Changing the values of $\theta_j$, the location of the roots changes smoothly.
Let $\theta_j$ be such that  there exists $y\in{\mathbb{R}}$ for which $Q(iy)=0$.
Then, there holds
\begin{equation*}
	\theta_1=-\sum_{j=0}^{k-1} \frac{\theta_{n+1-j}}{(1+iy)^{k-j}}
		=-\sum_{\ell=1}^{k} \frac{\theta_{j+1}}{(1+iy)^{\ell}}.
\end{equation*}
Then, we infer the estimate
\begin{equation*}
	\theta_1\leq \sum_{\ell=1}^{k} \frac{\theta_{j+1}}{|1+iy|^{\ell}}
		=\sum_{\ell=1}^{k} \frac{\theta_{j+1}}{(1+y^2)^{\ell/2}}
		\leq \sum_{\ell=1}^{k} \theta_{j+1}.
\end{equation*}
In particular, as soon as the condition \eqref{bigtheta1} is satisfied, 
no roots of the polynomial $Q$ may intersect the imaginary axis.
\end{proof}

As follows from the above analysis of the Routh--Hurwitz arrays, the above condition is not sharp.
In fact, in the case $k=3$, with $\theta_2=\theta_3=0$ and $\theta_4=1-\theta_1$, 
high frequency stability is equivalent to the requirement $\theta_1>1/9$, while
\eqref{bigtheta1} becomes $\theta_1>1/2$.
Similarly,  for $k=4$, with $\theta_2=\theta_3=\theta_4=0$ and $\theta_5=1-\theta_1$, 
high frequency stability is equivalent to $\theta_1>1/5$, while \eqref{bigtheta1} 
still corresponds to $\theta_1>1/2$.

\section{Intermediate frequency analysis}\label{sec:inter}

Next, we pass to the analysis of the dispersion relation \eqref{dispersion}
(with $p$ explicitly given in \eqref{disprel}) in the intermediate frequency regime,
under the assumption that the system is stable in the large frequency regime.
In particular, we assume $\theta_2<\theta_1$.

For $\lambda\neq -1$, the relation can be rewritten as
\begin{equation*}
	\lambda=\mu^2\sum_{j=1}^{k+1} \frac{\theta_{j}}{(\lambda+1)^{j}}
\end{equation*}
Let $(i\zeta,i\xi)$ with $\xi\neq 0$ be such that the above equality is satisfied;
then there hold
\begin{equation}\label{realimag}
	\sum_{j=1}^{k+1} \frac{\theta_j}{(1+\zeta^2)^{j}} {\textrm{Re}}\bigl((1-i\zeta)^{j}\bigr)=0,
	\qquad
	\zeta+\xi^2 \sum_{j=1}^{k+1} \frac{\theta_j}{(1+\zeta^2)^{j}}{\textrm{Im}}\bigl((1-i\zeta)^{j}\bigr)=0,
\end{equation}
so that, given $\zeta$ satisfying the first relation, $\xi$ is such that
\begin{equation}\label{xi2}
	\xi^2=-\left\{\sum_{j=1}^{k+1} \frac{\theta_j}{(1+\zeta^2)^{j}}
			{\textrm{Im}}\left(\frac{(1-i\zeta)^{j}}{\zeta}\right)\right\}^{-1}
\end{equation}
if and only if the term on the righthand side is positive.
Both the real part of $(1-i\zeta)^{j}$ and the imaginary part of $(1-i\zeta)^{j}/\zeta$
contain only even power of $\zeta$; 
hence the above relations can be rewritten in term of the single variable $s=\zeta^2$.

\begin{proof}[Proof of Theorem \ref{thm:anyk}]
We claim that if \eqref{condanyk} holds, then the first relation in \eqref{realimag} has no solutions.
Indeed, the first condition in \eqref{realimag} can be rewritten as
\begin{equation}\label{firstcond}
	\theta_1+\sum_{j=2}^{k+1} \frac{\theta_j}{(1+\zeta^2)^{j-1}} 
			{\textrm{Re}}\bigl((1-i\zeta)^{j}\bigr)=0
\end{equation}			
Since
\begin{equation*}
	\begin{aligned}
	\left|\sum_{j=2}^{k+1} \frac{\theta_j}{(1+\zeta^2)^{j-1}} 
			{\textrm{Re}}\bigl((1-i\zeta)^{k}\bigr)\right|
	&\leq \sum_{j=2}^{k+1} \frac{\theta_j}{|1+\zeta^2|^{j-1}} 
			\left|(1-i\zeta)^{j}\right|\\
	&=\sum_{j=2}^{k+1} \frac{\theta_j}{|1+\zeta^2|^{(j-2)/2}} 
		\leq \sum_{j=2}^{k+1} \theta_j,
	\end{aligned}
\end{equation*}	
under the condition \eqref{condanyk},
the relation \eqref{firstcond} cannot be satisfied.
\end{proof}

\subsection*{Combinations of the first three Gamma distributions}\label{sec:k=2}

Condition \eqref{condanyk}, stated in Theorem \ref{thm:anyk}, is sufficient, but not necessary.
Here, we restrict the attention to the case $k=2$, that is we concentrate on kernels $g$ of the form 
\begin{equation}\label{kernelk=2}
	g(t)=\left(\theta_1+\theta_2\,t+\tfrac{1}{2}\theta_3\,t^2\right)e^{-t}.
\end{equation}
with the aim of determining the sharp threshold
separating stability and instability.
Indeed, in this situation, it is possible to determine a complete characterization of the choices of the
coefficients $\theta_1, \theta_2, \theta_3$ giving raise to stable dynamics,
as encoded in Theorem \ref{thm:k=2}.

\begin{proof}[Proof of Theorem \ref{thm:k=2}]
Setting $s=\zeta^2$, the first equality in \eqref{realimag} becomes
\begin{equation*}
	\begin{aligned}
	0&=\theta_1(1+\zeta^2)^{2}{\textrm{Re}}\bigl(1-i\zeta\bigr)
		+\theta_2(1+\zeta^2){\textrm{Re}}\bigl((1-i\zeta)^{2}\bigr)
		+\theta_3{\textrm{Re}}\bigl((1-i\zeta)^{3}\bigr)\\
	&=\theta_1(1+s)^{2}+\theta_2(1+s)(1-s)+\theta_3(1-3s)
	\end{aligned}
\end{equation*}	
that is
\begin{equation}\label{defesse}
	(\theta_1-\theta_2)s^2+2\bigl(\theta_1-\tfrac32\theta_3\bigr)s
		+\theta_1+\theta_2+\theta_3=0
\end{equation}
Assuming $\theta_2<\theta_1$ (necessary condition for high-frequency stability),
equation \eqref{defesse} has positive roots if and only if
\begin{equation*}
	2\theta_1-3\theta_3<0\qquad\textrm{and}\qquad
	\bigl(\theta_1-\tfrac32\theta_3\bigr)^2-(\theta_1-\theta_2)(\theta_1+\theta_2+\theta_3)>0.
\end{equation*}
which can be equivalently rewritten as
\begin{equation}\label{unstabcond}
	3\theta_3<2\theta_1
	\qquad\textrm{and}\qquad
	(2\theta_2+\theta_3)^2+8(\theta_3-\theta_1)^2<8\theta_1^2.
\end{equation}
If these conditions are satisfied, for $s$ satisfying \eqref{defesse},
the expression \eqref{xi2} for $\xi^2$
reduces to
\begin{equation*}
	\xi^2=\frac{(1+s)^2}{\theta_2(1+s)+2\theta_3}.
\end{equation*}
having used the relation $\theta_1(1+s)^{2}+\theta_2(1+s)(1-s)+\theta_3(1-3s)=0$.

Thus, the unknown $\xi$ is explictly given by
\begin{equation*}
	\xi=\pm \frac{1+s}{\{\theta_2(1+s)+2\theta_3\}^{1/2}},
\end{equation*}
where $s=\zeta^2$ is a positive root of \eqref{defesse}.
\end{proof}

A graphical representation of the different stability regimes relative to the values of
parameter $\theta_1,\theta_2,\theta_3$ is given in Figure \ref{fig:stabregions1},
where the regions lie in the plane $(\eta_2,\eta_3)$ with $\eta_j:=\theta_j/\theta_1$.

In the analysis of diffusion equations with memory,
a typical requirement on the kernel $g$ concerns with its monotonicity and convexity properties.
Thus, we continue the study of the case \eqref{kernelk=2} determining the parameter range
in which such assumptions are satisfied.
Let us set
\begin{equation*}
	\mathcal{S}:=\{0\leq \eta_2\leq 1\}\cap
	\left(\{\eta_3\leq \tfrac{2}{3}\}\cup\Bigl\{\tfrac{1}{8}(2\eta_2+\eta_3)^2+(\eta_3-1)^2\leq 1\right\}\Bigr)
\end{equation*}
Our aim is to compare $\mathcal{S}$ with the regions $\mathcal{M}$ and $\mathcal{C}$
of monotonicity and mononotonicity + convexity, respectively.
\vskip.25cm

\noindent{\bf Monotonicity.}
There holds
\begin{equation*}
	g'(t)=-\tfrac12\left\{2(1-\eta_2)+2(\eta_2-\eta_3)t+\eta_3\,t^2\right\}\theta_1\,e^{-t}
\end{equation*}
where, as usual, $\eta_k=\theta_k/\theta_1$.
Then, $g$ is decreasing if and only if the polynomial
\begin{equation*}
	p_1(t):=2(1-\eta_2)+2(\eta_2-\eta_3)t+\eta_3\,t^2
\end{equation*}
is positive for any $t>0$.
Since $\eta_3>0$, a necessary condition is $p_1(0)=2(1-\eta_2)>0$.
If this is satisfied, then the condition holds if and only if
\begin{equation*}
	\textrm{either}\quad \eta_2-\eta_3\geq 0
	\qquad\textrm{or}\quad
	\left\{\begin{aligned}
	 &\eta_2-\eta_3\leq 0,\\ & (\eta_2-\eta_3)^2-2(1-\eta_2)\eta_3\leq 0.
	\end{aligned}\right.
\end{equation*}
Collecting, we deduce that, in the plane $(\eta_2,\eta_3)$,
the parameters region in which $g$ is decreasing is
\begin{equation*}
	\mathcal{M}:=\{0\leq \eta_2\leq 1,\,\eta_3\geq 0\}\cap
		\Bigl(\{\eta_2-\eta_3\geq 0\}\cup\{\eta_2^2+(\eta_3-1)^2\leq 1\}\Bigr).
\end{equation*}
\noindent{\bf Convexity.}
Differentiating, we infer
\begin{equation*}
	g''(t)=\tfrac12\left\{2(1-2\eta_2+\eta_3)+2(\eta_2-2\eta_3)t+\eta_3\,t^2\right\}\theta_1\,e^{-t}
\end{equation*}
Then, $g$ is convex if and only if the polynomial
\begin{equation*}
	p_2(t):=2(1-2\eta_2+\eta_3)+2(\eta_2-2\eta_3)t+\eta_3\,t^2
\end{equation*}
is positive for any $t>0$.
Again, since $\eta_3>0$, a necessary condition is $p_2(0)=2(1-2\eta_2+\eta_3)>0$.
If this holds, function $g$ is convex if and only if
\begin{equation*}
	\textrm{either}\quad \eta_2-2\eta_3\geq 0
	\qquad\textrm{or}\quad
	\left\{\begin{aligned}
	 &\eta_2-2\eta_3\leq 0,\\ & (\eta_2-2\eta_3)^2-2(1-2\eta_2+\eta_3)\eta_3\leq 0.
	\end{aligned}\right.
\end{equation*}
Hence, the region $\mathcal{C}$ in the plane $(\eta_2,\eta_3)$ is
the parameters region in which $g$ is decreasing is
\begin{equation*}
	\mathcal{C}:=\{\eta_2,\,\eta_3\geq 0,\,2\eta_2-\eta_3\leq 1\}\cap
	\Bigl(\{\eta_2-2\eta_3\geq 0\}\cup\left\{2\eta_2^2+4(\eta_3-1/2)^2\leq 1\right\}\Bigr)
\end{equation*}
\begin{figure}
\leftline{\hskip1cm\includegraphics[width=4cm]{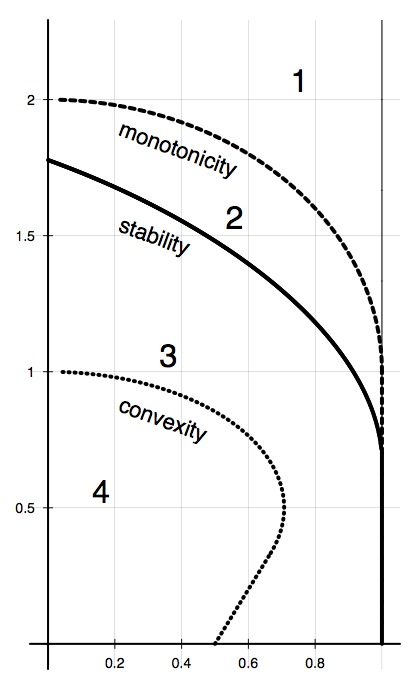}}
\vskip-6.25cm	\hskip5.25cm\includegraphics[width=8cm]{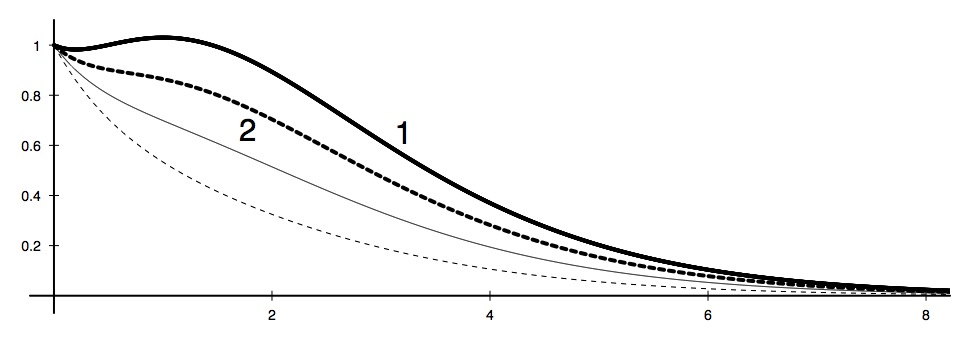}\\
\vskip.15cm	\hskip5.25cm \includegraphics[width=8cm]{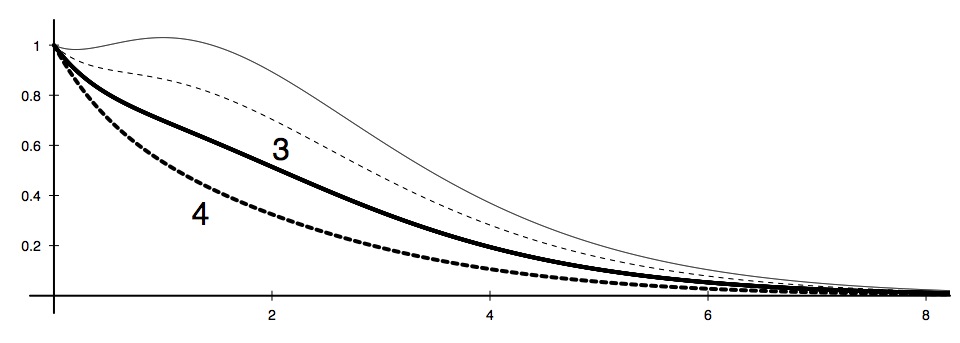}
\caption{\footnotesize 
Left: Regions  $\mathcal{S}$, $\mathcal{C}$ and $\mathcal{M}$ in the plane $(\eta_2,\eta_3)$.
Right: Graph of the kernels corresponding to points 1, 2, 3 and 4.
Top-right:		point 1, $(\eta_2,\eta_3)=(0.8, 2.0)$ (non-monotone, unstable);
			point 2, $(\eta_2,\eta_3)=(0.6,1.5)$ (monotone, non-convex, unstable).
Bottom-right:	point 3, $(\eta_2,\eta_3)=(0.4,1)$ (monotone, non-convex, stable);
			point 4, $(\eta_2,\eta_3)=(0.2, 0.5)$ (monotone, convex, stable).}
\label{fig:SMC}
\end{figure}

\noindent{\bf Relation with stability.}
By inspection, it can be seen that 
\begin{equation*}
	\mathcal{C}\subset \mathcal{S}\subset \mathcal{M}.
\end{equation*}
A graphic representation of the regions $\mathcal{S}$, $\mathcal{C}$ and $\mathcal{M}$
is provided in Figure \ref{fig:SMC}.

Thus, for kernels of the form \eqref{kernelk=2}, stability is guaranteed by convexity.
This is coherent with the result in \cite{ContMarcPata13},
taking into account that the function $\mu$ corresponds to $-g'$.
Indeed, such result is relative to convex kernels ($\mu$ non-increasing)
and apply to the present case since the ``flatness rate'' (see details in the above reference)
for kernel of the form \eqref{kernelk=2} is always zero.

Also, still in the case of kernel of the form \eqref{kernelk=2}, non-monotonicity implies instability.
Differently, we can state that monotonicity is a necessary condition for stability for kernels 
\eqref{kernelk=2}, consistently with Theorem 1.5 in \cite{ChepMainPata06} (proved for a
general class of memory kernels).

Finally, the analysis exhibits a large region of decreasing non-convex
kernels divided into two parts by the stability/instability transition line.
In particular, there exist (monotone) kernels $g$ which are stable even
if not convex.

{\small
\subsection*{Acknowledgements}
The author is thankful to Maurizio Grasselli and Vittorino Pata for some useful (electronic)
discussions on the topic.
This work was partially supported by the Project FIRB 2012
``Dispersive dynamics: Fourier Analysis and Variational Methods''.
}

\end{document}